\documentclass[review,onefignum,onetabnum]{siamart171218}

\usepackage{amsfonts}
\usepackage{amsmath} 
\usepackage{lipsum}
\usepackage{stmaryrd}
\usepackage{graphicx}
\usepackage{epstopdf}

\usepackage{algorithm2emod}
\ifpdf
  \DeclareGraphicsExtensions{.eps,.pdf,.png,.jpg}
\else
  \DeclareGraphicsExtensions{.eps}
\fi

\newsiamremark{remark}{Remark}
\newsiamremark{hypothesis}{Hypothesis}
\crefname{hypothesis}{Hypothesis}{Hypotheses}
\newsiamthm{claim}{Claim}

\headers{Strong stability of RK time discretizations}{Z. Sun and C.-W. Shu}

\title{Strong Stability of Explicit Runge--Kutta Time Discretizations\thanks{Submitted to the editors on Nov 25, 2018. This work is mainly done during the symposium held at ICERM in Providence RI, for celebrating 75 years of Mathematics of Computation. We would like to thank the committee members, speakers and staffs for this inspiring workshop. ZS especially wants to express his appreciation for the travel and lodging support from ICERM.}}

\author{Zheng Sun\thanks{Department of Mathematics, The Ohio State University,
		Columbus, OH 43210
  (\email{sun.2516@osu.edu}).}
\and Chi-Wang Shu\thanks{Division of Applied Mathematics, Brown University,
	Providence, RI 02912 
  (\email{shu@dam.brown.edu}). Research supported by ARO grant W911NF-16-1-0103
  and NSF grant DMS-1719410.}}

\usepackage{amsopn}

\newcommand{\hf}{\frac{1}{2}}

\newcommand{\dt}{\frac{d}{dt}}
\newcommand{\dtau}{\frac{d}{d\tau}}
\newcommand{\ddtau}[1]{\frac{d^{#1}}{d\tau^{#1}}}
 
\newcommand{\veps}{\varepsilon}
\newcommand{\nmH}[1]{\| #1 \|_H}
\newcommand{\nm}[1]{\| #1 \|}
\newcommand{\qnmH}[1]{\llbracket #1 \rrbracket_H}
\newcommand{\ipH}[1]{\langle #1 \rangle_H}
\newcommand{\qipH}[1]{[ #1 ]_H}
\newcommand{\qip}[1]{[ #1 ]}
\newcommand{\ip}[1]{\langle #1 \rangle}
\newcommand{\Mod}[1]{\ (\mathrm{mod}\ #1)}

\newcommand{\cE}{\mathcal{E}}

\newcommand{\cO}{\mathcal{O}}

\ifpdf
\hypersetup{
  pdftitle={Strong Stability of Explicit Runge--Kutta Time Discretizations},
  pdfauthor={Zheng Sun and Chi-Wang Shu}
}
\fi

\begin{document}

\maketitle

\begin{abstract}
 Motivated by studies on fully discrete numerical schemes for linear hyperbolic conservation laws, we present a framework on analyzing the strong stability of explicit Runge--Kutta (RK) time discretizations for semi-negative autonomous linear systems. The analysis is based on the energy method and can be performed with the aid of a computer. Strong stability of various RK methods, including a sixteen-stage embedded pair of order nine and eight, has been examined under this framework. Based on numerous numerical observations, we further characterize the features of strongly stable schemes. A both necessary and sufficient condition is given for the strong stability of RK methods of odd linear order. 
\end{abstract}



\section{Introduction}
Explicit Runge--Kutta (RK) methods have been commonly used for time integration of hyperbolic conservation laws. In this paper, we study the strong stability of the methods in such context. We focus on autonomous linear ordinary differential equation (ODE) systems
\begin{equation}\label{eq-odes}
\frac{d}{dt}u = Lu,
\end{equation}
which are obtained from method of lines schemes for linear hyperbolic problems. $u\in \mathbb{R}^N$ and $L$ is an $N\times N$ real constant matrix, where $N$ is the degrees of freedom for the spatial discretization.  If the semidiscrete scheme honors the (weighted) $L^2$ stability of the partial differential equation (PDE), then for certain symmetric and positive definite matrix $H$, 
\begin{equation}\label{eq-sn}
L^\top H + H L \leq 0
\end{equation} 
is a semi-negative definite matrix. Here $H$ can be related with both the symmetrizer  of the PDE \cite{gustafsson1995time} and the mass matrix or quadrature weights of a Galerkin or collocation type spatial discretization. If \eqref{eq-sn} holds, then we say $L$ is semi-negative and \eqref{eq-odes} satisfies the energy decay law
\begin{equation}\label{eq-energy}
\begin{aligned}
\dt \nmH{u}^2 &= \ipH{\dt u, u}+\ipH{u, \dt  u}\\
&= \ip{Lu, H u} + \ip{u,H L u} = \ip{u,(L^\top H+H L)u} \leq 0.
\end{aligned}
\end{equation}
Here $\ipH{\cdot, \cdot} = \ip{\cdot, H \cdot}$ with $\ip{\cdot,\cdot}$ 
being the usual $l^2$ inner product in $\mathbb{R}^N$ and $\nmH{u} 
= \sqrt{\ipH{u,u}}$. We are concerned with whether this property is 
preserved at the discrete level, namely, whether
\begin{equation}\label{eq-strong}
\nmH{u^{n+1} } \leq \nmH{u^n},
\end{equation}
holds after applying an explicit RK time integrator under a suitably restricted time step. The time step constraint is referred to as the Courant--Friedrichs--Lewy (CFL) condition for numerical conservation laws. 
We would say the explicit RK method is strongly stable,
if \eqref{eq-strong} is satisfied when discretizing (\ref{eq-odes})
under the condition (\ref{eq-sn}). Note this definition of stability is 
stronger than the usual one, which states that the norm of the
numerical solution is bounded by the norm of the initial data up 
to a constant (typically dependent on the total time) \cite{gustafsson1995time}. We remark that the concept of strong stability in our context is also connected with the conditional contractivity from the ODE community \cite{spijker1983contractivity,kraaijevanger1991contractivity}, while our analysis does not rely on the circle condition, which is usually assumed in the previous works. 

Solving hyperbolic conservation laws with explicit time discretizations can be dated back to early days. They seem to be the most natural choice to avoid the inversion of large nonlinear systems resulted from spatial discretizations. Many of the first order monotone schemes, which preserves various properties of the continuous PDE, adopt Euler forward method for time marching \cite{randall1992numerical}. This flavor has been pushed further with the development of strong-stability-preserving (or total-variation-diminishing) high order time discretizations \cite{shu1988total,shu1988efficient,gottlieb1998total,gottlieb2001strong}, especially the RK methods in this family. The strong-stability-preserving RK (SSPRK) methods can be formulated as convex combinations of Euler forward steps, hence automatically preserve many properties that are achieved by the Euler forward discretization. One of the successes is to use them for integrating method of lines scheme obtained from discontinuous Galerkin (DG) spatial discretizations for hyperbolic conservation laws. The method is referred to as the RKDG method, which was developed by Cockburn et al in a series of papers \cite{rkdg1,rkdg2,rkdg3,rkdg4,rkdg5}. In the RKDG method, a limiter can be applied after each Euler forward stage to control total variation. The non-increasing total variation semi-norm is then preserved after a full time step with SSPRK methods. Another application of SSPRK methods is for designing high order positivity-preserving or bound-preserving schemes \cite{zhang2010maximum,zhang2011maximum}. The approach, which is based on a similar methodology, has received intensive attentions in recent years and has been successfully applied to various problems  \cite{zhang2010positivity,xing2010positivity,wu2015high,qin2016bound,sun2018discontinuous,sun2018entropy}. It should also be noted, although progress has been made on preserving positivity with the Euler backward method \cite{qin2018implicit}, due to the nonexistence of second or higher order methods as the positive combinations of Euler backward steps \cite{gottlieb2001strong}, it is a nontrivial task to design high order implicit positivity-preserving schemes.

Along with the popularity of explicit RK methods for hyperbolic problems, a growing attention has been paid to the role that the time integrator has played in a fully discrete scheme. One of the major issues is on the stability: whether the $L^2$ stability achieved by method of lines schemes would be preserved after the explicit time stepping.  Indeed, this issue has been raised during the initial development of the RKDG method. Even for linear advection equation, it is reported in \cite{chavent1989local} that under the usual CFL condition, the Euler forward time stepping coupled with the linear DG method is unstable, even in the sense of the weaker stability. The second order RK methods are stable only if the spatial discretization uses at most piecewise linear elements \cite{cockburn2001runge}. From then on, analyses particularly on fully discrete DG methods have been performed, such as \cite{zhang2010stability,zhang2012fully,wang2015stability,wang2016stability, sun2017stability,zhou2018stability}. But 
it seems that an universal approach of analyzing the stability is missing.

One of the attempts is to perform eigenvalue analysis. If the matrix $L$ is normal, then it can be diagonalized with an orthogonal matrix. Hence the ODE system is transformed into decoupled scalar equations. The method is strongly stable if and only if the amplifier of each decoupled equation falls within the absolute stability region of the time integrator. However, when $L$ is not normal, which is generic for $L$ obtained from method of lines schemes, eigenvalue analysis gives only necessary but not sufficient conditions on strong stability. In such situation, eigenvalue analysis should be avoided for the following reasons. Firstly, if $L$ can be diagonalized, the analysis ensures the strong stability of the transformed system. But in practice, the diagonalizing matrix can be ill-conditioned, the resulting stability would have very weak controls on the actual solution after the backward transformation \cite{levy1998semidiscrete}. Secondly, to ensure stability, we need the matrix norm of the amplifier to be properly controlled, but the eigenvalue analysis only examines the spectral radius, which is strictly smaller than the matrix norm for non-normal cases. Hence the actual time step constraint should be stricter than that from the eigenvalue analysis. The result solely relying on eigenvalue analysis can be misleading. We refer to \cite{iserles2009first} for a particular example.

A more sound way of analyzing stability is to use the energy method. The study of \eqref{eq-odes}  through this approach initiates from problems with coercive matrices $L^\top H + HL \leq -\eta L^\top H L$, where $\eta$ is a positive constant. In \cite{levy1998semidiscrete}, Levy and Tadmor proved that for coercive matrices, the classic third order and fourth order explicit RK methods are strongly stable, under the time step constraint $\tau \leq \lambda \eta$ with $\lambda = \frac{3}{50}$ for the third order method and $\lambda = \frac{1}{62}$ for the fourth order method. Later, a simpler proof was discovered based on the strong-stability-preserving RK formulation and the fact that Euler forward method is strongly stable for the 
coercive problems \cite{gottlieb2001strong}. This approach gives a significantly relaxed time step $\tau \leq \eta$ and extends the result to linear RK methods of arbitrary order. These analyses coincide with the earlier contractivity analysis in \cite{spijker1983contractivity} and \cite{kraaijevanger1991contractivity}. In their study of contractivity, or strong stability in our context, a circle condition is assumed, which is essentially equivalent to the strong stability assumption for the Euler forward steps. 

The coercivity condition typically arises from diffusive problems, but is uncommon in numerical approximations of hyperbolic problems. This motivates us to remove this assumption and consider the general semi-negative case.  The difficulty is that the operator becomes less dissipative and the Euler forward step is no longer strongly stable. In \cite{tadmor2002semidiscrete}, Tadmor proved the third order RK method is strongly stable for semi-negative $L$. Then in \cite{sun2017rk4}, we found a counter example showing that the fourth order RK method can not preserve the strong stability, although the fourth order method satisfies the necessary condition in the eigenvalue analysis and is hence strongly stable if $L$ is normal. Furthermore, in the same paper, the fourth order RK method is proven to be strongly stable in two steps. In other words, successively applying the method for two steps yields a strongly stable method with eight stages. Recently in \cite{ranocha2018l_2}, it is shown that the low storage SSPRK method of order four with ten stages also admits strong stability. 

As one can see from \cite{sun2017rk4} and \cite{ranocha2018l_2}, although the computation for analyzing RK methods with many stages looks complicated at the first glance, it only involves elementary algebraic manipulations.  A single technique, which corresponds to integration by parts for the PDE, has been repeatedly used. Inspired by the proofs in  \cite{sun2017rk4} and \cite{ranocha2018l_2}, as well as many previous analyses on DG methods, we develop a unified framework on analyzing the strong stability of explicit RK methods. The main idea is based on an induction procedure. With the aid of a computer, we can easily obtain the energy equality, which equates the energy at the next time steps with that at the current time level plus terms of specific forms. Then a sufficient condition, generalized from Lemma 2.4 in \cite{sun2017rk4}, is used to justify the strong stability. A necessary condition is also provided to exclude some methods that are not strongly stable. With this framework, we can easily examine strong stability of various RK methods, including a ninth order and eighth order embedded pair with sixteen stages. Finally, base on numerous observations, we summarize patterns in the energy equality to further characterize the strongly stable RK methods. In particular, we give a both necessary and sufficient condition for the RK methods of odd linear order to be strongly stable. We remark that the framework has its limitation, for example, it can not determine if the classic fourth order method is strongly stable.

The rest of the paper is organized as follows. In \cref{sec-2}, we present our framework on analyzing the strong stability of explicit RK methods of any order with arbitrary stages. In \cref{sec-3}, the strong stability of various methods has been examined using the framework, including linear RK methods, the classic fourth order methods (strong stability in multiple steps), several SSPRK methods and the embedded pairs used in the commercial software Mathematica. Then in \cref{sec-4}, we further investigate the energy equality to characterize the structure of strongly stable methods. Numerical dissipation of these methods for energy conserving systems has also been discussed. Finally in \cref{sec-5}, conclusions are given. 

\section{Stability analysis: a framework}\label{sec-2}
Let us drop the superscript $n$ in $u^n$. Consider an explicit RK time discretization for the linear autonomous system \eqref{eq-odes}. The scheme is of the form 
\begin{equation}\label{eq-rkscheme}
u^{n+1} = R_s u,
\end{equation}
where 
\begin{equation}\label{eq-rk}
\qquad R_s = \sum_{ k = 0}^s \alpha_s (\tau L)^k, \qquad \alpha_0 = 1, \qquad \alpha_s \neq 0.
\end{equation}
Here $\tau$ is the time step and $s$ is the number of stages. The coefficients $\{\alpha_k\}_{k=0}^s$ depend solely on the scheme itself. For an $s$-stage method, it is of linear order $p$ if and only if the first $p+1$ terms 
in the summation \eqref{eq-rk} coincide with the truncated Taylor series of $e^{\tau L}$. In particular, $p\leq s$ \cite{butcher2016numerical}.
We would like to examine the strong stability of \eqref{eq-rkscheme} under the usual CFL condition: if there exists a constant $\lambda$, such that 
\begin{equation}\label{eq-Rsulu}
\nmH{R_s u} ^2\leq \nmH{u}^2,
\end{equation}
for all $\tau \nmH{L}\leq \lambda$ and all inputs $u$. This is equivalent to
\begin{equation}
\nmH{R_s} \leq 1,
\end{equation}
under the prescribed condition and $\nmH{R_s}$ is the matrix norm of $R_s$.

A natural attempt is to adopt the following expansion to compare $\nmH{R_su}^2$ with $\nmH{u}^2$.
\begin{equation}\label{eq-exps1}
\nmH{R_su}^2 = \sum_{i ,j= 0}^s \alpha_i \alpha_j\tau^{i+j} \ipH{L^iu,L^ju} = \nmH{u}^2 + \sum_{1\leq\max\{i ,j\}\leq s} \alpha_i \alpha_j \tau^{i+j}\ipH{L^iu,L^ju}.
\end{equation}
However, each term $\ipH{L^iu,L^ju}$ may not necessarily have a sign. The idea for overcoming the difficulty is to convert $\ipH{L^i u, L^ju}$ into linear combinations of terms of the form $\nmH{L^k u}^2$, $\qnmH{L^k u}^2$ and $\qipH{L^i u, L^j u}$. Here 
\begin{equation}
\qipH{v,w} = -\ip{v, (L^\top H + H L) w}, \qquad v,w\in \mathbb{R}^N
\end{equation}
is a semi inner product and
\begin{equation}
\qnmH{v} =\sqrt{\qipH{v,v}}
\end{equation}
defines the induced semi-norm. Indeed, this can be achieved through the following induction procedure. 

\begin{proposition}[Integration by parts]\label{prop-routine}
	Suppose $j\geq i$, then
	\begin{equation}\label{eq-prop-1}
	\ipH{L^i u, L^j u} = \left\{
	\begin{matrix}
	&\nmH{L^i u }^2,&\qquad j = i,\\
	&-\frac{1}{2} \qnmH{L^i u }^2,&\qquad j = i+1,\\
	&-\ipH{L^{i+1} u, L^{j-1} u } - \qipH{L^iu,L^{j-1}u},&\qquad \text{otherwise.}
	\end{matrix}
	\right.
	\end{equation}
\end{proposition}
\begin{proof}
	The case $j = i$ can be justified with the definition of $\nmH{\cdot}$. When $j = i+1$, 
	\begin{equation}
	\begin{aligned}
	\ipH{L^i u , L^{i+1} u} &= \hf \ip{L^i u, (HL) L^i u} + \hf \ip{L^\top H L^i u, L^i u} \\
	&= \hf\ip{L^i u, (L^\top H + HL) L^i u} = -\frac{1}{2}\qnmH{L^i u}^2.
	\end{aligned}
	\end{equation}
	When $j > i+1$, 
	\begin{equation}
	\begin{aligned}
	\ipH{L^i u , L^{j} u} &= \ip{L^i u, (HL) L^{j-1} u} \\
	&=  - \ip{L^i u, L^\top H L^{j-1} u} 
	+ \ip{L^i u, (L^\top H + HL) L^{j-1} u}\\
	&= -\ipH{L^{i+1} u, L^{j-1} u } - \qipH{L^iu,L^{j-1}u}.
	\end{aligned}
	\end{equation}
\end{proof}

In the context of approximating the spatial derivative $\partial_x$ for periodic functions with $L$, \cref{prop-routine} is the discrete version of integration by parts. Since $L$ may not preserve the exact anti-symmetry of $\partial_x$, namely $L^\top H + HL \neq 0$, an extra term $\qipH{L^iu,L^{j-1}u}$ is produced. Furthermore, $-\frac{1}{2}\qnmH{L^iu}^2$ is usually the numerical dissipation from the spatial discretization. In particular, $\qnmH{\cdot}$ is the jump semi-norm in the DG method.

Furthermore, one can repeat the induction procedure to obtain the following expansion.
\begin{corollary}\label{cor-induction}
	For $j\geq i$, 
	\begin{equation}\label{eq-prop-2}
	\ipH{L^iu,L^ju} =
	\zeta_{i,j} - \sum_{k=0}^{\lfloor{\frac{j-i}{2}}\rfloor -1}(-1)^{k}\qipH{L^{i+k}u,L^{j-1-k}u},
	\end{equation}
	where
	\begin{equation}
	\zeta_{i,j} = \left\{
	\begin{array}{cc}
	(-1)^{\frac{j-i+1}{2}}\frac{1}{2}\qnmH{L^{\frac{i+j-1}{2}}u}^2,&\qquad i+j \text{ odd},\\ (-1)^{\frac{j-i}{2}}\nmH{L^{\frac{i+j}{2}}u}^2,&\qquad i+j \text{ even}.\\
	\end{array}\right.
	\end{equation}
\end{corollary}

Based on \cref{cor-induction}, we have the following energy equality.

\begin{lemma}[Energy equality]\label{lem-exps}
	Given $H$ and $R_s = \sum_{k=0}^s \alpha_k (\tau L)^k$ with $\alpha_0 = 1$. There exists a unique set of coefficients $\{\beta_k\}_{k=0}^s\cup \{\gamma_{i,j}\}_{i,j=0}^{s-1}$, such that for all $u$ and $L$ satisfying $L^\top H + H L \leq 0$, 
	\begin{equation}\label{eq-exps}
	\nmH{R_s u}^2 = \sum_{k = 0}^s\beta_k \tau^{2k} \nmH{L^k u}^2 + \sum_{i,j = 0}^{s-1}\gamma_{i,j} \tau^{i+j+1}\qipH{L^i u, L^j u}, \qquad \gamma_{i,j} = \gamma_{j,i}.
	\end{equation}
\end{lemma}

\begin{proof}
	The existence of such expansion can be justified by \eqref{eq-exps1} and \cref{cor-induction}. It remains to show that the set of coefficients is unique. Each term in the summation should be considered as a polynomial of $\tau$, elements in $L$ and elements in $u$. We are going to show these polynomials are linearly independent.
	
	(i) It suffices to analyze the case $H$ being the identity matrix, otherwise we can consider $\tilde{L} = \sqrt{H}L\sqrt{H}^{-1}$ and $\tilde{u} = \sqrt{H} u$ instead. 
	
	(ii) It suffices to show $\{\tau^{i+j}\ip{L^iu, L^j u}\}_{0\leq i\leq j \leq s}$ are linearly independent. Since elements in the set can be expressed as linear combinations of $\{\tau^{2k}\nm{L^ku}^2\}_{k=0}^s\cup \{\tau^{i+j}\qip{L^i u,L^j u}\}_{0\leq i\leq j\leq s-1 }$ due to \cref{cor-induction}. Since the two sets have the same cardinality, the linear independence of the previous set implies that of the latter one. 
	
	Particularly, we take $L = \left(\begin{matrix}
	\cos \theta & -\sin \theta\\
	\sin \theta & \cos \theta
	\end{matrix}\right)$ with $\theta\in(\frac{\pi}{2}, \frac{3\pi}{2})$ (hence $L^\top + L\leq 0$) and $u = \left(\begin{matrix}
	1\\0\end{matrix}\right)$. Noting that $L$ is a rotation matrix, which is orthogonal, one can obtain
	\begin{equation}
	\begin{aligned}
	0 = \sum_{0\leq i\leq j \leq s} \alpha_{i,j} \tau^{i+j}\ip{L^iu,L^j u} = \sum_{m = 0}^{2s}\left(\sum_{k = \max\{0,m-s\}}^{\lfloor\frac{m}{2}\rfloor}\alpha_{k,m-k} \ip{L^k u, L^{m-k}u}\right)\tau^m\\
	=  \sum_{m = 0}^{2s}\sum_{k = \max\{0,m-s\}}^{\lfloor\frac{m}{2}\rfloor}\alpha_{k,m-k} \cos((m-2k)\theta)\tau^m.
	\end{aligned}
	\end{equation}
	Note $\{m -2k\}_{k = \max\{0,m-s\}}^{\lfloor\frac{m}{2}\rfloor}$ are distinct non-negative integers. Due to linear independence of $\cos((m-2k)\theta)\tau^m$, $\{\alpha_{k,m-k}\}_{k=\max\{0,m-s\}}^{\lfloor\frac{m}{2}\rfloor}$ are all zeros for each $m$. Hence $\{\tau^{i+j}\ip{L^iu, L^j u}\}_{0\leq i\leq j \leq s}$ are linearly independent.
\end{proof}
\begin{remark}
	The uniqueness is not used in the framework. But it facilitates our analysis in \cref{sec-4}. Note that the uniqueness is nontrivial. For example, if we restrict ourselves to a small subset $\{L: L^\top H + HL = -\eta L^{\top} H L, 0>\eta \in \mathbb{R} \}$, then one can certainly obtain different linear combinations. 
\end{remark}

To facilitate our discussion, we introduce the following definitions.
\begin{definition}
	The \underline{leading index} of $R_s$, denoted as $k^*$, is the positive integer such that $\beta_{k^*} \neq 0$ and $\beta_k < 0 $ for all $1\leq k < k^*$. The coefficient $\beta_{k^*}$ is called the \underline{leading coefficient}. The $k^*$-th order principal submatrix $\Gamma^{*} = (\gamma_{i,j})_{0\leq i,j \leq {k^*-1}}$ is called the \underline{leading submatrix}.
\end{definition}
Note that $k^*$ is well-defined since $\beta_{s} = \alpha_s^2 \neq 0$, which implies $k^*\leq s$. For small $\tau \nmH{L}$, $\beta_{k^*}\tau^{2k^*}\nmH{L^k u}^2$ and $\sum_{i,j \geq 0}^{k^*-1}\gamma_{i,j} \tau^{i+j+1}\qipH{L^iu,L^ju}$ become dominant terms in the energy equality. Hence the strong stability would be closely related with the negativity of $\beta_{k^*}$ and $\Gamma^*$. In particular, we have the following necessary condition and sufficient condition. 

\begin{theorem}[Necessary condition]\label{thm-nec}
	The method is not strongly stable if $\beta_{k^*}>0$. More specifically, if $\beta_{k^*} > 0 $, then there exists a constant $\lambda$, such that $\nmH{R_s}> 1$ if $0< \tau \nmH{L}\leq \lambda$ and $L^\top H + H L = 0$. 
\end{theorem}
\begin{proof}
	With $L^\top H + HL = 0$, the latter summation of terms $\qipH{L^iu,L^ju}$ in \eqref{eq-exps} is zero. Hence
	\begin{equation}
	\begin{aligned}
	\nmH{R_s u}^2 &= \nmH{u}^2 + \beta_{k^*}\tau^{2k^*}\nmH{L^{k^*}u}^2 + \sum_{k = {k^*+1}}^s\beta_k \tau^{2k} \nmH{L^ku}^2\\
	&\geq \nmH{u}^2 + \left(\beta_{k^*}- \sum_{k = {k^{*}+1}}^{s}|\beta_{k}| (\tau\nmH{L})^{2(k-k^*)}\right)
	\tau^{2k^*} \nmH{L^{k^*}u}^2.\\
	\end{aligned}
	\end{equation}
	Therefore, 
	\begin{equation}
	\nmH{R_s}^2\geq 1 + (\beta_{k^{*}}-\tilde{\beta})\tau^{2k}\nmH{L^{k^*}}^2>1,
	\end{equation}
	if $\tilde{\beta} = \sum_{k =k^*+1}^{s}|\beta_{k}|\lambda^{2(k-k^*)}<\beta_{k^*}$.
\end{proof}
\begin{theorem}[Sufficient condition]\label{thm-suff}
	If $\beta_{k^*} < 0 $ and $\Gamma^*$ is negative definite, then there exists a constant $\lambda$ such that $\nmH{R_s } \leq 1$ if $\tau \nmH{L} \leq \lambda$. 
\end{theorem}
\begin{proof}
	Let $-\veps$ to be the largest eigenvalue of $\Gamma^*$. Then $\Gamma^* +\veps I$ is negative semi-definite. From Lemma 2.3 in \cite{sun2017rk4}, $\sum_{i,j = 0}^{k^*-1}(\gamma_{i,j} + \veps \delta_{i,j}) \tau^{i+j+1}\qipH{L^iu,L^ju}\leq 0$, where $\delta_{i,j}$ is the Kronecker delta function. Hence
	\begin{equation}
	\begin{aligned}
	\nmH{R_s u}^2 \leq& \nmH{u}^2 + \beta_{k^*}\tau^{2k^*} \nmH{L^{k^*} u}^2 + \sum_{k = k^*+1}^s\beta_k \tau^{2k} \nmH{L^k u}^2\\
	&- \veps\sum_{k = 0}^{k^*-1}\tau^{2k+1}\qnmH{L^ku}^2 + \sum_{k^*\leq \max\{i, j\} \leq s-1 }\gamma_{i,j} \tau^{i+j+1}\qipH{L^i u, L^j u}.
	\end{aligned}
	\end{equation}
	Note that $\nmH{L^k u} \leq \nmH{L}^{k-k^*}\nmH{L^{k^*}u}$ and $\tau \nmH{L}\leq \lambda$. Hence we have
	\begin{equation}
	\sum_{k = k^*+1}^s\beta_k \tau^{2k} \nmH{L^k u}^2\leq  \left(\sum_{k = k^*+1}^s|\beta_k|\lambda^{2(k-k^*)}\right)\tau^{2k^*} \nmH{L^{k^*} u}^2.
	\end{equation}
	Using the fact
	\begin{equation}
	\qipH{L^iu,L^ju} \leq \qnmH{L^iu}\qnmH{L^ju}, \quad 
	\tau\qnmH{L^ju}^2 \leq 2\lambda \nmH{L^j u}^2\leq 2\lambda \nmH{L}^{2(j-k^*)}\nmH{L^{k^*}u}^2,
	\end{equation}
	together with the arithmetic-geometric mean inequality, one can obtain
	\begin{equation}
	\tau^{i+j+1}\qipH{L^iu,L^ju}\leq
	\left\{
	\begin{matrix}
	\frac{\veps\tau^{2i+1}}{2\tilde{\beta}}\qnmH{L^iu}^2 + \frac{\tilde{\beta}\lambda^{2(j-k^*)+1}}{\veps}\tau^{2k^*}\nmH{L^{k^*}u}^2, &\quad i< k^*, j \geq k^*,\\
	2\lambda^{i+j+1-2k^*}\tau^{2k^*}\nmH{L^{k^*}u}^2,& \quad  i, j \geq k^*.\\
	\end{matrix}\right.
	\end{equation}
	Therefore,
	\begin{equation}
	\begin{aligned}
	\nmH{R_s u}^2 \leq&  \nmH{u}^2 + \left(\beta_{k^*}+\sum_{k = k^*+1}^s|\beta_k|\lambda^{2(k-k^*)} +2\sum_{i,j\geq k^*}^s |\gamma_{i,j}|\lambda^{i+j+1-2k^*}\right.\\
	&\left. + \frac{2\tilde{\beta}}{\veps}\sum_{j= k^*}^s\sum_{i=1}^{k^*-1} |\gamma_{i,j}|\lambda^{2(j-k^*)+1}\right)\tau^{2k^*} \nmH{L^{k^*} u}^2  \\
	& -\veps \sum_{k =0}^{k^*-1}\big(1-\frac{\sum_{j = k^*}^s |\gamma_{k,j}|}{\tilde{\beta}}\big)\tau^{2k+1}\qnmH{L^k u}^2.\\
	\end{aligned}
	\end{equation}
	It suffices to take $\tilde{\beta} = \max\{\sum_{j = k^*}^s |\gamma_{k,j}|\}_{k = 0}^{k^*-1}$ and then choose $\lambda$ sufficiently small so that the second coefficient on the right is negative. 
\end{proof}

Given an $s$-stage RK scheme $R_s$, we expand $\nmH{R_su}^2$ with \cref{lem-exps} and then use the necessary condition in \cref{thm-nec} and sufficient condition in \cref{thm-suff} to examine its strong stability. Note that $\{\beta_{k}\}_{k=0}^s$ and $\{\gamma_{i,j}\}_{i,j = 0}^{s-1}$ can be obtained from \cref{alg-energy}, which is based on \cref{prop-routine}. Since $\Gamma^*$ is symmetric, one only needs to check its largest eigenvalue to determine if $\Gamma^*$ is negative definite. 
		\begin{algorithm}[H]
			\SetAlgoLined
			\caption{Obtain coefficients in \cref{lem-exps}}\label{alg-energy}
			\KwIn{$\beta = (\beta_0,\cdots,\beta_s) = 0$, $\Gamma = (\gamma_{i,j})_{i,j = 0}^{s-1} = 0$}
			
			\For{$i \leftarrow 0$ to $s$}{
				$\beta_i\leftarrow\beta_i + \alpha_i^2$\\
				\For{$j \leftarrow i+1$ to $s$}{
					$\tilde{\alpha} \leftarrow 2\alpha_i\alpha_j$\\
					$k \leftarrow i$\\
					$l \leftarrow j$\\
					\While{$l>k$}{
						\Switch{$l$}{\Case{$k$}{$\beta_{k} \leftarrow \beta_{k} + \tilde{\alpha}$}\Case{$k+1$}{$\gamma_{k,k} \leftarrow \gamma_{k,k} - \tilde{\alpha}/2$}\Other{$\gamma_{k,l-1} \leftarrow \gamma_{k,l-1} - \tilde{\alpha}$\\
								$\tilde{\alpha} \leftarrow -\tilde{\alpha}$
						}}
						$k \leftarrow k+1$\\
						$l \leftarrow l -1$
					}
				}
			}
		\end{algorithm}

\section{Applications}\label{sec-3}

In this section, we examine strong stability of various RK methods using the framework in the previous section. In tables that will be provided later, the last column ``SS" refers to the strong stability property. A question mark will be put into the entry if strong stability of the corresponding scheme can not be determined. ``no*" means particular counter examples can be constructed. 

\subsection{Linear RK methods}
For general nonlinear systems, to admit accuracy order higher than four, RK methods must have more stages than its order \cite{butcher2016numerical}. However, for autonomous linear systems, the desired order of accuracy can be achieved with the same number of stages. All such methods would be equivalent to the Taylor series method
\begin{equation}
R_p = P_p = \sum_{k = 0}^p\frac{(\tau L)^k}{k!}.
\end{equation}
In \cref{tab-LRK1} and \cref{tab-LRK2}, we document the leading indexes and coefficients of linear RK methods from first order to twelfth order. The leading submatrices for low order RK methods are also given in \cref{tab-LRK1}. We report that the third order, seventh order and eleventh order methods are strongly stable. The fourth order, eighth order and twelfth order methods can not be judged with the framework. All other methods are not strongly stable.
\begin{table}[h!]
	\centering
	\small
	\begin{tabular}{|c|ccccc|}
		\hline
		$p$&$k^*$&$\beta_{k^*}$&$\Gamma^*$&$\lambda(\Gamma^*)$&SS\\
		\hline
		1&1&1&$-\left(
		\begin{array}{c}
		1 \\
		\end{array}
		\right)$&$
		\begin{array}{l}
		-1.00000\ \quad \qquad \\
		\end{array}
		$&no\\
		\hline
		2&2&$\frac{1}{4}$&$-\left(
		\begin{array}{cc}
		1 & \frac{1}{2} \\
		\frac{1}{2} & \frac{1}{2} \\
		\end{array}
		\right)$&$\begin{array}{l}
		-1.30902 \\
		-1.90983\times 10^{-1} \\
		\end{array}$&no\\
		\hline
		3&2&$-\frac{1}{12}$&$-\left(
		\begin{array}{cc}
		1 & \frac{1}{2} \\
		\frac{1}{2} & \frac{1}{3} \\
		\end{array}
		\right)$&$\begin{array}{l}
		-1.26759 \\
		-6.57415\times 10^{-2} \\
		\end{array}$&yes\\
		\hline
		4&3&$-\frac{1}{72}$&$-\left(\begin{array}{ccc}
		1 & \frac{1}{2} & \frac{1}{6} \\
		\frac{1}{2} & \frac{1}{3} & \frac{1}{8} \\
		\frac{1}{6} & \frac{1}{8} & \frac{1}{24} \\
		\end{array}\right)$&$\begin{array}{l}
		-1.30128 \\
		-7.93266\times 10^{-2} \\
		+5.60618\times 10^{-3} \\
		\end{array}$&no*\\
		\hline
		5&3&$\frac{1}{360}$&$-\left(
		\begin{array}{ccc}
		1 & \frac{1}{2} & \frac{1}{6} \\
		\frac{1}{2} & \frac{1}{3} & \frac{1}{8} \\
		\frac{1}{6} & \frac{1}{8} & \frac{1}{20} \\
		\end{array}
		\right)$&$\begin{array}{l}
		-1.30150 \\
		-8.07336\times 10^{-2} \\
		-1.10151\times 10^{-3} \\
		\end{array}$&no\\
		\hline
		6&4&$\frac{1}{2880}$&$-\left(
		\begin{array}{cccc}
		1 & \frac{1}{2} & \frac{1}{6} & \frac{1}{24} \\
		\frac{1}{2} & \frac{1}{3} & \frac{1}{8} & \frac{1}{30} \\
		\frac{1}{6} & \frac{1}{8} & \frac{1}{20} & \frac{1}{72} \\
		\frac{1}{24} & \frac{1}{30} & \frac{1}{72} & \frac{1}{240} \\
		\end{array}
		\right)$&$\begin{array}{l}
		-1.30375 \\
		-8.21871\times 10^{-2} \\
		-1.40529\times 10^{-3} \\
		-1.60133\times 10^{-4} \\
		\end{array}$&no\\
		\hline
		7&4&$-\frac{1}{20160}$&$-\left(
		\begin{array}{cccc}
		1 & \frac{1}{2} & \frac{1}{6} & \frac{1}{24} \\
		\frac{1}{2} & \frac{1}{3} & \frac{1}{8} & \frac{1}{30} \\
		\frac{1}{6} & \frac{1}{8} & \frac{1}{20} & \frac{1}{72} \\
		\frac{1}{24} & \frac{1}{30} & \frac{1}{72} & \frac{1}{252} \\
		\end{array}
		\right)$&$\begin{array}{l}
		-1.30375 \\
		-8.21836\times 10^{-2} \\
		-1.36301\times 10^{-3} \\
		-7.86229\times 10^{-6} \\
		\end{array}$&yes\\
		\hline
		8&5&$-\frac{1}{201600}$&$-\left(\begin{array}{ccccc}
		1 & \frac{1}{2} & \frac{1}{6} & \frac{1}{24} & \frac{1}{120} \\
		\frac{1}{2} & \frac{1}{3} & \frac{1}{8} & \frac{1}{30} & \
		\frac{1}{144} \\
		\frac{1}{6} & \frac{1}{8} & \frac{1}{20} & \frac{1}{72} & \
		\frac{1}{336} \\
		\frac{1}{24} & \frac{1}{30} & \frac{1}{72} & \frac{1}{252} & \
		\frac{1}{1152} \\
		\frac{1}{120} & \frac{1}{144} & \frac{1}{336} & \frac{1}{1152} & \
		\frac{23}{120960} \\
		\end{array}\right)$&$\begin{array}{l}
		-1.30384 \\
		-8.22588\times 10^{-2} \\
		-1.38580\times 10^{-3} \\
		-9.32706\times 10^{-6} \\
		+2.24989\times 10^{-6} \\
		\end{array}$&?\\
		\hline
	\end{tabular}
	\caption{Linear RK methods: from first order to eighth order.}\label{tab-LRK1}
\end{table}
\begin{table}
	\centering
	\small
	\begin{tabular}{|c|cccc|}
		\hline
		$p$&$k^*$&$\beta_{k^*}$&$\lambda(\Gamma^*)$&SS\\
		\hline
		9&5&$\frac{1}{1814400}$&$\begin{array}{l}
		-1.30384 \\
		-8.22588\times 10^{-2} \\
		-1.38585\times 10^{-3} \\
		-9.75366\times 10^{-6} \\
		-3.11800\times 10^{-8} \\
		\end{array}$&no\\
		\hline
		10&6&$\frac{1}{221772800}$&$\begin{array}{l}
		-1.30384 \\
		-8.22613\times 10^{-2} \\
		-1.38688\times 10^{-3} \\
		-9.91006\times 10^{-6} \\
		-4.70638\times 10^{-8} \\
		-1.63872\times 10^{-8} \\
		\end{array}$&no\\
		\hline
		11&6&$-\frac{1}{239500800}$&$\begin{array}{l}
		-1.30384 \\
		-8.22613\times 10^{-2} \\
		-1.38688\times 10^{-3} \\
		-9.90966\times 10^{-6} \\
		-3.87351\times 10^{-8} \\
		-7.87018\times 10^{-11} \\
		\end{array}$&yes\\
		\hline
		12&7&$-\frac{1}{3353011200}$&$\begin{array}{l}
		-1.30384 \\
		-8.22614\times 10^{-2} \\
		-1.38691\times 10^{-3} \\
		-9.91617\times 10^{-6} \\
		-3.93334\times 10^{-8} \\
		+1.45458\times 10^{-10} \\
		-8.54170\times 10^{-11} \\
		\end{array}$&?\\
		\hline
	\end{tabular}
	\caption{Linear RK methods: from ninth order to twelfth order.}\label{tab-LRK2}
\end{table}

\subsection{The classic fourth order method}
The classic fourth order method with four stages, which is widely used in practice due to its stage and order optimality, is unfortunately not covered under the framework. In \cite{sun2017rk4}, we found a counter example to show that the method is not strongly stable, but successively applying the method for two steps yields 
a strongly stable method with eight stages. 
\begin{proposition}[Sun and Shu, 2018]\label{prop-rk4}
	The fourth order RK method with four stages is not strongly stable. More specifically, for $H = I$ and $L = 
	-\left(\begin{array}{ccc}
	1&2&2\\0&1&2\\0&0&1\end{array}\right)$, we have $\nm{P_4}>1$, if $\tau \nmH{L}>0$ is sufficiently small.
\end{proposition}
\begin{theorem}[Sun and Shu, 2018]
	The fourth order RK method with four stages is strongly stable in two steps. In other words, there exists a constant $\lambda$, such that $\nmH{(P_4)^2}\leq 1$ if $\tau \nmH{L}\leq \lambda$.
\end{theorem}
Here we examine multi-step strong stability of the fourth order method using our framework. Note the derivation using this framework is slightly different from that in \cite{sun2017rk4}. The relevant quantities for strong stability are given in \cref{tab-rk4}. Note that the method is both two-step and three-step strongly stable (with the same time step size), which means the norm of the solution after the first step is always bounded by the initial data, if sufficiently small uniform time steps are used.
\begin{table}[h!]
	\centering
	\small
	\begin{tabular}{|c|ccccc|}
		\hline
		$(P_4)^m$&$k^*$&$\beta_{k^*}$&$\Gamma^*$&$\lambda(\Gamma^*)$&SS\\
		\hline
		$(P_4)^2$&3&$-\frac{1}{36}$&$-\left(
		\begin{array}{ccc}
		2 & 2 & \frac{4}{3} \\
		2 & \frac{8}{3} & 2 \\
		\frac{4}{3} & 2 & \frac{19}{12} \\
		\end{array}
		\right)$&$\begin{array}{l}
		-5.73797 \\
		-4.99093\times 10^{-1} \\
		-1.29329\times 10^{-2} \\
		\end{array}$&yes\\
		\hline
		$(P_4)^3$&3&$-\frac{1}{24}$&
		$-\left(
		\begin{array}{ccc}
		3 & \frac{9}{2} & \frac{9}{2} \\
		\frac{9}{2} & 9 & \frac{81}{8} \\
		\frac{9}{2} & \frac{81}{8} & \frac{97}{8} \\
		\end{array}
		\right)$&$\begin{array}{l}
		-2.28380\times 10^1 \\
		-1.21069 \\
		-7.62892\times 10^{-2} \\
		\end{array}$&yes\\
		\hline
	\end{tabular}
	\caption{The classic fourth order method: multi-step strong stability.}\label{tab-rk4}
\end{table}
\begin{theorem}\label{thm-RK4n} 
	The four-stage fourth order RK method has the following property. With uniform time steps such that $\tau \nmH{L}\leq {\lambda}$ for sufficiently small $\lambda$, $\nmH{u^{n}}\leq\nmH{u^0}$ for all $n>1$. 
\end{theorem}

\subsection{SSPRK methods}
In this section, we study the strong stability of several SSPRK methods. The (explicit) SSPRK methods are a class of RK methods that can be formulated as combinations of Euler forward steps. The second order method with two stages, and the third order method with three stages are equivalent to the linear RK methods for autonomous linear systems, which have been discussed. For fourth order methods, in order to avoid backward-in-time steps and negative coefficients, at least five stages should be used \cite{gottlieb1998total,kraaijevanger1991contractivity}, which are denoted as SSPRK(5,4). We specially consider the method that is independently discovered in \cite{kraaijevanger1991contractivity} and \cite{spiteri2002new}. When applied on \eqref{eq-odes}, the method takes the form 
\begin{equation}
\text{SSPRK(5,4)} = P_4+4.477718303076007\times 10^{-3} (\tau L)^5.
\end{equation} 
We will study its strong stability (in multiple steps) using our framework. Besides SSPRK(5,4), we will also consider two commonly used low storage SSPRK methods, a third order method with four stages SSPRK(4,3) and a fourth order method with ten stages SSPRK(10,4) \cite{kraaijevanger1991contractivity, spiteri2002new}. The two methods are 
\begin{equation}
\text{SSPRK(4,3)} = P_3 + \frac{1}{48}(\tau L)^4,
\end{equation}
and 
\begin{equation}
\begin{aligned}
\text{SSPRK(10,4)} =& P_4
+\frac{17}{2160}(\tau L)^5 + \frac{7}{6480}(\tau L)^6 + \frac{1}{9720}(\tau L)^7 \\
&+ \frac{1}{155520}(\tau L)^8 + \frac{1}{4199040}(\tau L)^9 + \frac{1}{251942400}(\tau L)^{10},
\end{aligned}
\end{equation}
when applied on \eqref{eq-odes}. We remark that the strong stability of SSPRK(10,4) has been proved by Ranocha and \:Offner in \cite{ranocha2018l_2}. This is a reexamination using our framework. 
\begin{table}[h!]
	\centering
	\small
	\begin{tabular}{|c|ccccc|}
		\hline
		SSPRK&$k^*$&$\beta_{k^*}$&$\Gamma^*$&$\lambda(\Gamma^*)$&SS\\
		\hline
		$\text{(4,3)}$&2&$-\frac{1}{24}$&$-\left(
		\begin{array}{cc}
		1 & \frac{1}{2} \\
		\frac{1}{2} & \frac{1}{3} \\
		\end{array}
		\right)$&$
		\begin{array}{l}
		-1.26759 \\
		-6.57415\times 10^{-2} \\
		\end{array}$&yes\\
		\hline
		$\text{(10,4)}$&3&$-\frac{1}{3240}$&$-\left(
		\begin{array}{ccc}
		1 & \frac{1}{2} & \frac{1}{6} \\
		\frac{1}{2} & \frac{1}{3} & \frac{1}{8} \\
		\frac{1}{6} & \frac{1}{8} & \frac{107}{2160} \\
		\end{array}
		\right)$&$\begin{array}{l}
		-1.30149 \\
		-8.06493\times 10^{-2} \\
		-7.35115\times 10^{-4} \\
		\end{array}$&yes\\
		\hline
		$\text{(5,4)}$&3&$-4.93345\times10^{-3}$&$-\left(
		\begin{array}{ccc}
		1 & \frac{1}{2} & \frac{1}{6} \\
		\frac{1}{2} & \frac{1}{3} & \frac{1}{8} \\
		\frac{1}{6} & \frac{1}{8} & \frac{1}{24} \\
		\end{array}
		\right)$&
		$\begin{array}{l}
		-1.30140 \\
		-8.00541\times 10^{-2} \\
		+1.97309\times 10^{-3} \\
		\end{array}$&no*\\
		\hline
		$\text{(5,4)}^2$&3&$-9.86690\times 10^{-3}$&$-\left(\begin{array}{ccc}
		2 & 2 & \frac{4}{3} \\
		2 & \frac{8}{3} & 2 \\
		\frac{4}{3} & 2 & 1.5923 \\
		\end{array}\right)$&$\begin{array}{l}
		-5.74021 \\
		-5.01739\times 10^{-1} \\
		-1.70056\times 10^{-2} \\
		\end{array}$&yes\\
		\hline
		$\text{(5,4)}^3$&3&$-1.48004\times 10^{-2}$&$-\left(
		\begin{array}{ccc}
		3 & \frac{9}{2} & \frac{9}{2} \\
		\frac{9}{2} & 9 & \frac{81}{8} \\
		\frac{9}{2} & \frac{81}{8} & 12.138 \\
		\end{array}
		\right)$&$
		\begin{array}{l}
		-2.28450\times 10^1 \\
		-1.21415 \\
		-7.93174\times 10^{-2} \\
		\end{array}$&yes\\
		\hline
	\end{tabular}
	\caption{SSPRK methods: strong stability and multi-step strong stability.}\label{tab-ssprk}
\end{table}
From \cref{tab-ssprk}, we are able to conclude the follow results.
\begin{theorem}
	SSPRK(4,3) and SSPRK(10,4) are strongly stable.
\end{theorem}
\begin{theorem} 
	The property stated in \cref{thm-RK4n} also holds for SSPRK(5,4).
\end{theorem}

The behavior of SSPRK(5,4) is very similar to that of the classic fourth order method, since it is almost the four-stage method except for a small fifth order perturbation. Although the method can not be judged within this framework, one can indeed use the same counter example in \cref{prop-rk4} to disprove its strong stability. 
The proof would be similar to that in \cite{sun2017rk4}.

\subsection{Embedded RK methods in NDSolve}
Finally, we consider embedded RK pairs that are used in NDSolve, a function for numerically solving differentiable equations in the commercial software Mathematica.  Embedded RK methods are pairs of RK methods sharing the same stages. The notation $p(\hat{p})$ is commonly used, if two methods in the pair are of order $p$ and order $\hat{p}$ respectively. Their Butcher tableau has the following form.
\begin{equation}
\begin{tabular}{c|cccccc}
0&&&&&\\
$c_2$&$a_{2,1}$&&&&\\
$c_{3}$&$a_{3,1}$&$a_{3,2}$&$$&&\\
$\vdots$&$\vdots$&$\vdots$&$\ddots$&&\\
$c_{s}$&$a_{s,1}$&$a_{s,2}$&$\cdots$&$a_{s,s-1}$&\\
\hline
&$b_1$&$b_2$&$\cdots$&$b_{s-1}$&$b_s$\\
&$\hat{b}_1$&$\hat{b}_2$&$\cdots$&$\hat{b}_{s-1}$&$\hat{b}_s$\\
\end{tabular}
\end{equation}
For $\frac{d}{dt}u= f(t,u)$, the tableau gives two solutions
\begin{equation}
u^{n+1} = u^n +\tau \sum_{i=1}^s  b_i k_i,\qquad \hat{u}^{n+1} = u^n + \tau\sum_{i=1}^s  \hat{b}_i k_i,
\end{equation}
where
\begin{equation}
k_i = f(t+c_i\tau, u_n + \tau \sum_{j = 1}^sa_{i,j} k_j).
\end{equation}
Then the difference $u^{n+1} - \hat{u}^{n+1}$ can be used for local error estimates for time step adaption. 

We examine strong stability of all such pairs used in Mathematica from order $2(1)$ to order $9(8)$.  These methods are chosen with several desired properties being considered, including the FSAL (First Same As Last) strategy and stiffness detection capability \cite{wolfram2008advanced}. Tableaux of 2(1), 3(2) and 4(3) pairs \cite{sofroniou2004construction} are given in \cref{tab-pair2}, \cref{tab-pair3} and \cref{tab-pair4}.
\begin{table}[ht!]
	\centering
	\begin{tabular}{c|ccc}
		0 & \text{} & \text{} & \text{} \\
		1 & 1 & \text{} & \text{} \\
		1 & $\frac{1}{2}$ & $\frac{1}{2}$ & \text{} \\
		\hline
		\text{} & $\frac{1}{2}$ & $\frac{1}{2}$ & 0 \\
		\text{} & 1 & $-\frac{1}{6}$ & $\frac{1}{6}$ \\
	\end{tabular}
	\hspace{\textwidth}
	\caption{Tableau of embedded RK $2(1)$.}\label{tab-pair2}
\end{table}
\begin{table}[ht!]
	\centering
	\begin{tabular}{c|cccc}
		0 & \text{} & \text{} & \text{} & \text{} \\
		$\frac{1}{2}$ & $\frac{1}{2}$ & \text{} & \text{} & \text{} \\
		1 & -1 & 2 & \text{} & \text{} \\
		1 & $\frac{1}{6}$ &$ \frac{2}{3}$ &$ \frac{1}{6}$ & \text{} \\
		\hline
		\text{} & $\frac{1}{6}$ & $\frac{2}{3} $& $\frac{1}{6}$ & 0 \\
		\text{} & $\frac{22-\sqrt{82}}{72}$&$ \frac{\sqrt{82}+14}{36} $ &$ \frac{\sqrt{82}-4}{144} $& $\frac{16-\sqrt{82}}{48} $ \\
	\end{tabular}
	\hspace{\textwidth}
	\caption{Tableau of embedded RK $3(2)$.}\label{tab-pair3}
\end{table}
\begin{table}[ht!]
	\centering
	\begin{tabular}{c|ccccc}
		0 & \text{} & \text{} & \text{} & \text{} & \text{} \\
		$	\frac{2}{5}$ & $\frac{2}{5} $& \text{} & \text{} & \text{} & \text{} \\
		$	\frac{3}{5}$ & $-\frac{3}{20}$ & $\frac{3}{4} $& \text{} & \text{} & \text{} \\
		1 &$ \frac{19}{44}$ & $-\frac{15}{44}$ & $\frac{10}{11} $& \text{} & \text{} \\
		1 & $\frac{11}{72}$&$ \frac{25}{72} $& $\frac{25}{72} $&$ \frac{11}{72} $& \text{} \\
		\hline
		\text{} & $\frac{11}{72}$ & $\frac{25}{72}$ & $\frac{25}{72}$ & $\frac{11}{72}$ & 0 \\
		\text{} & $\frac{1251515}{8970912} $&$ \frac{3710105}{8970912} $& $\frac{2519695}{8970912}$ & $\frac{61105}{8970912}$ & $\frac{119041}{747576}$ \\
	\end{tabular}
	\hspace{\textwidth}
	\caption{Tableau of embedded RK $4(3)$.}\label{tab-pair4}
\end{table}
For the 5(4) pair \cite{bogacki1996efficient,shampine2018numerical} and higher order pairs \cite{verner2010numerically}, the tableaux can be obtained through the Mathematica command
\begin{equation}
\text{NDSolve\`{}EmbeddedExplicitRungeKuttaCoefficients[}p\text{, Infinity]}.
\end{equation}
The output takes the form
\begin{equation}
\{A,b,c,b-\hat{b}\}
\end{equation}
and the corresponding tableau is
\begin{equation}
\begin{tabular}{c|c}
$c$&$A$\\
\hline
&$b$\\
&$\hat{b}$
\end{tabular}.
\end{equation}
The stability results are documented in \cref{tab-NDSolve1} and \cref{tab-NDSolve2}.

\begin{table}
	\centering
	\small
	\begin{tabular}{|c|c|ccccc|}
		\hline
		Methods&$s$&$p$/$\hat{p}$&$k^*$&$\beta_{k^*}$&$\lambda(\Gamma^*)$&SS\\
		\hline
		2(1)&3&2&2&$\frac{1}{4}$&
		$\begin{array}{l}
		-1.30902\\
		-1.90983\times 10^{-1}
		\end{array}$&no\\
		\cline{3-7}
		&&1&1&1&$
		\begin{array}{ll}
		-1.00000&\  \qquad\\
		\end{array}$&no\\
		\hline
		3(2)&4&3&2&$-\frac{1}{12}$&
		$\begin{array}{l}-1.26759\\-6.57415\times 10^{-2}\end{array}$&yes\\
		\cline{3-7}
		&&2&2&$\frac{1}{12}$&
		$\begin{array}{l}
		-1.28130 \\
		-1.11257\times 10^{-1} \\
		\end{array}$&no\\
		\hline
		4(3)&5&4&3&$-\frac{1}{72}$&$\begin{array}{l}
		-1.30128 \\
		-7.93266\times 10^{-2} \\
		+5.60618\times 10^{-3} \\
		\end{array}$&no*\\
		\cline{3-7}
		&&3&2&$-\frac{119041}{4485456}$&$\begin{array}{l}
		-1.26759 \\
		-6.57415\times 10^{-2} \\
		\end{array}$&yes\\
		\hline
		5(4)&8&5&3&$-\frac{43}{6209280}$&$
		\begin{array}{l}
		-1.3015 \\
		-8.07336\times 10^{-2} \\
		-1.10151\times 10^{-3} \\
		\end{array}$&yes\\
		\cline{3-7}
		&&4&3&$\frac{51767}{367590960}$&$\begin{array}{l}
		-1.30150 \\
		-8.07430\times 10^{-2} \\
		-1.14174\times 10^{-3} \\
		\end{array}$&no\\
		\hline
		6(5)&9&6&4&$\frac{79007}{2560896000}$&
		$\begin{array}{l}
		-1.30375 \\
		-8.21839\times 10^{-2} \\
		-1.36689\times 10^{-3} \\
		-2.38718\times 10^{-5} \\
		\end{array}$&no\\
		\cline{3-7}
		&&5&3&$\frac{1233467}{9027158400}$&
		$\begin{array}{l}
		-1.30150 \\
		-8.07336\times 10^{-2} \\
		-1.10151\times 10^{-3} \\
		\end{array}$&no\\
		\hline
		7(6)&10&7&4&$\frac{29615605063}{38967665360400000}$&
		$\begin{array}{l}
		-1.30375 \\
		-8.21836\times 10^{-2} \\
		-1.36301\times 10^{-3} \\
		-7.86229\times 10^{-6} \\
		\end{array}$&no\\
		\cline{3-7}
		&&6&4&$-\frac{20202919901}{1855603112400000}$&$
		\begin{array}{l}
		-1.30375 \\
		-8.21833\times 10^{-2} \\
		-1.35985\times 10^{-3} \\
		+5.49402\times 10^{-6} \\
		\end{array}$&?\\
		\hline
	\end{tabular}
	\hspace{\textwidth}
	\flushleft
	\begin{footnotesize}
		*The fourth order method in the 4(3) pair is exactly the classic four-stage fourth order method for autonomous linear systems.
	\end{footnotesize}
	\caption{Embedded RK pairs: from $2(1)$ to $7(6)$ pairs.} \label{tab-NDSolve1}
\end{table}
\begin{table}
	\centering
	\small
	\begin{tabular}{|c|c|ccccc|}
		\hline
		Methods&s&p&$k^*$&$\beta_{k^*}$&$\lambda(\Gamma^*)$&SS\\
		\hline
		8(7)&13&8&5&$-3.21308\times 10^{-7}$&
		$\begin{array}{l}
		-1.30384 \\
		-8.22588\times 10^{-2} \\
		-1.38584\times 10^{-3} \\
		-9.71236\times 10^{-6} \\
		+1.43671\times 10^{-7} \\
		\end{array}$&?\\
		\cline{3-7}
		&&7&4&$-2.39706\times 10^{-6}$&$\begin{array}{l}
		-1.30375 \\
		-8.21836\times 10^{-2} \\
		-1.36301\times 10^{-3} \\
		-7.86229\times 10^{-6} \\
		\end{array}$&yes\\
		\hline
		9(8)&16&9&5&$-8.95352\times 10^{-9}$&$\begin{array}{l}
		-1.30384 \\
		-8.22588\times 10^{-2} \\
		-1.38585\times 10^{-3} \\
		-9.75366\times 10^{-6} \\
		-3.11800\times 10^{-8} \\
		\end{array}$&yes\\
		\cline{3-7}
		&&8&5&$-5.46447\times 10^{-7}$&$\begin{array}{l}
		-1.30384 \\
		-8.22588\times 10^{-2} \\
		-1.38585\times 10^{-3} \\
		-9.78641\times 10^{-6} \\
		-1.64476\times 10^{-7} \\
		\end{array}$&yes\\
		\hline
	\end{tabular}
	\caption{Embedded RK pairs: $8(7)$ pair and $9(8)$ pair.} \label{tab-NDSolve2}
\end{table}

\section{Characterization of strongly stable methods}\label{sec-4}

By numerically examining various RK methods, we recognize certain patterns of the coefficients in \cref{lem-exps}, which will be proved in this section. 

To characterize the coefficients, we would like to split the RK operator as a truncated exponential and a high order perturbation. 
\begin{equation}
R_s = P_p + (\tau L)^{p+1} Q_{s-(p+1)},
\end{equation}
where
\begin{equation}\label{eq-rksplit}
P_p = \sum_{k=0}^p \frac{(\tau L)^k}{k!},\qquad
Q_{s-(p+1)} = \sum_{k=0}^{s-(p+1)} \alpha_{k+p+1} {(\tau L)^k}, \quad \alpha_{p+1} \neq \frac{1}{(p+1)!}.
\end{equation}
Note that an RK method of order $p$ for general nonlinear system may achieve higher order accuracy when applied to linear autonomous problems. Without special clarification, the order $p$ in this section refers to the linear order. 

\subsection{Coefficients in the energy equality}
Coefficients in \cref{lem-exps} have the following pattern.

\begin{lemma}\label{lem-key}
	For an $s$-stage RK method of linear order $p$, $\beta_0 = 1$. Furthermore, \\
	(i) if $p$ is odd, then
	\begin{equation}
	k^*= \frac{ p + 1}{2},\quad\beta_{k^*} =  (-1)^{k^*}2{\left(\alpha_{p+1}-\frac{1}{(p+1)!}\right)},
	\end{equation}
	\begin{equation}
	\gamma_{i,j} = -\frac{1}{i!j!(i+j+1)},  \quad  \forall 0\leq i,j\leq k^*-1;
	\end{equation}
	(ii) if $p$ is even, then \\
	\begin{equation}
	k^*\geq \frac{ p}{2}+1,\quad \beta_{\frac{ p}{2}+1} = (-1)^{\frac{p}{2}+1}2\left(\alpha_{p+2}- \alpha_{p+1}+\frac{1}{p!(p+2)}\right),
	\end{equation}
	\begin{equation}
	\gamma_{i,j} = -\frac{1}{i!j!(i+j+1)}  + \iota_{i,j,p}  \quad  \forall 0\leq i,j\leq \frac{p}{2},
	\end{equation}
	where
	\begin{equation}
	\iota_{i,j,p} = \left\{
	\begin{matrix}
	(-1)^{\frac{p}{2}+1}\left(\alpha_{p+1} - \frac{1}{(p+1)!}\right),&\quad i = j = \frac{p}{2},\\
	0,&\quad \text{otherwise.}\end{matrix}\right.
	\end{equation}
\end{lemma}
The proof of this lemma is postponed to the end of the section. 

\subsection{Criteria for strong stability}

With \cref{lem-key}, one can obtain the following theorem regarding strong stability of linear RK methods with
$R_s = P_p$. Note that the case $p\equiv 0 \Mod 4$ is not covered. The difficulty for analyzing this class of methods has already been recognized in \cite{sun2017rk4}.
\begin{theorem}	Consider a linear RK method of order $p$ with $p$ stages.\\
	(i) The method is not strongly stable if $p \equiv 1 \Mod 4$ or $p\equiv 2\Mod 4$. \\
	(ii) The method is strongly stable if $p\equiv 3 \Mod 4$.
\end{theorem}
\begin{proof}
	Note $\alpha_{p+1}= \alpha_{p+2}=0$. $\beta_{k^*}>0$ if $p \equiv 1 \Mod 4$ or $p\equiv 2\Mod 4$. The method can not preserve strong stability due to \cref{thm-nec}. On the other hand, for $p\equiv 3\Mod 4$, the leading submatrix $\Gamma^*$ can be written as $-\Lambda M \Lambda$, where $\Lambda = \text{diag}(1/0!,1/1!,1/2!,\cdots,1/(k^*-1)!)$ and $M$ is the Hilbert matrix of order $k^*$.  $\Gamma^*$ is negative definite since the Hilbert matrix is positive definite. Also note $\beta_{k^*} = -\frac{2}{(p+1)!}$ in such cases. The strong stability can be proved using \cref{thm-suff}. 
\end{proof}

\begin{remark}
	This modulo pattern with periodicity $4$ is closely related with the fact that $i^4 = 1$, in which $i$ is the imaginary unit. One can get a flavor by considering the scalar ODE $\frac{d}{dt} u =( i \omega) u$, $\omega \in \mathbb{R}$. 
\end{remark}

For a method with non-zero $Q_{s-(p+1)}$, 
noting that only limited number of stages would affect the leading coefficient and submatrix, one can conclude the following criteria for strong stability. We highlight that the condition for methods of odd linear order is both necessary and sufficient.

\begin{theorem}\label{thm-odd}
	An RK method of odd linear order $p$ is strongly stable if and only if  
	\begin{equation}\label{eq-thmodd}
	(-1)^{\frac{p+1}{2}}\left(\alpha_{p+1} -\frac{1}{(p+1)!}\right) <0.
	\end{equation}
\end{theorem}

\begin{proof}
	Since $\Gamma^* = -\Lambda M\Lambda$ is always negative definite, the method is strongly stable if and only if $\beta_{k^*}<0$, which corresponds to the prescribed condition in the theorem.
\end{proof}

For RK methods with even linear order, we can only obtain a sufficient condition, which is given in \cref{thm-even}. A similar condition has also been discussed in \cite{ranocha2018l_2}  for $p = 4$. 

\begin{theorem}\label{thm-even}An RK method of even linear order $p$ is strongly stable if 
	\begin{equation}\label{eq-thm-even-1}
	(-1)^{\frac{p}{2}+1}\left(\alpha_{p+2} -\alpha_{p+1}+\frac{1}{p!(p+2)}\right)<0,
	\end{equation}
	and
	\begin{equation}\label{eq-thm-even-2}
	(-1)^{\frac{p}{2}+1}\left(\frac{p}{2}!\right)^2\left(\alpha_{p+1} -\frac{1}{(p+1)!}\right)<\veps.
	\end{equation} 
	Here $\veps$ is the smallest eigenvalue of the Hilbert matrix of order $\frac{p}{2}+1$.
\end{theorem}
\begin{proof}
	From \cref{lem-key}, \eqref{eq-thm-even-1}  implies $k^* = \frac{p}{2}+1$ and $\beta_{k^*}<0$. Note $\Gamma^*$ is negative definite with \eqref{eq-thm-even-2}. Strong stability then follows from \cref{thm-suff}.
\end{proof}

We remark that, constraints in \cref{thm-odd} and \cref{thm-even} can be used with order conditions for designing strongly stable RK methods. 

\subsection{Regarding energy conserving systems}
Problems for wave propagations are usually featured with a conserved $L^2$ energy. This conservation is also expected numerically to maintain the accurate shape and phase of the waves in long time simulations. Suppose an energy-conserving spatial discretization is used, for example \cite{chung2009optimal,cheng2017l2,fu2018optimal,fu2019energy}, the resulting method of lines scheme would satisfy $L^\top H + H L =0$. Hence $\frac{d}{dt}\nmH{u}^2 = 0$. While a strongly stable RK time discretization may not preserve this equality, we would like know how the total energy is dissipated. The following interpretation can be obtained based on \cref{lem-exps}. While one can also perform eigenvalue analysis alternatively, since $L$ is normal in $\ipH{\cdot,\cdot}$ for energy conserving systems.

With $L^\top H + HL = 0$, $\qipH{\cdot,\cdot} = 0$. The energy equality in \cref{lem-exps} would then become
\begin{equation}
\nmH{R_s u }^2 = \nmH{u}^2 + \sum_{k = k^*}^{s}\beta_k\tau^{2k}\nmH{L^ku}^2.
\end{equation}
Since $\beta_{k^*}<0$ for strongly stable RK methods, we have dissipative energy unless $L = 0$. 

\begin{proposition}
	With a suitably restricted time step, a strongly stable explicit RK method is energy conserving, if and only if $L = 0$. 
\end{proposition}

At the final time $T = n \tau$, $\nmH{u^n}^2 = \nmH{u^0}^2 + \cO(\tau^{2k^*-1})$. The total numerical dissipation due to the time integrator would then be of order $2k^*-1$. We refer to $2k^*-1$ as the energy accuracy of the RK time integrator. With $k^*$ determined in \cref{lem-key}, we have the following proposition. 

\begin{proposition}\label{prop-ec}
	Consider an RK method of linear order $p$ applied on an energy conserving system.\\
	(i) The order of energy accuracy is $p$ if $p$ is odd.\\
	(ii) The order of energy accuracy is at least ${p}+1$ if $p$ is even. It achieves higher energy accuracy if and only if $\alpha_{p+2} = \alpha_{p+1} - \frac{1}{p!(p+2)}$.
\end{proposition}

One can see from \cref{prop-ec}, the linearly even order RK methods achieve at least one degree higher order of energy accuracy than we usually expect. Hence compared with the odd order methods, they may be more suitable for integrating systems modeling wave propagations. This also reveals the fact that even order methods introduce less numerical dissipation, and explains why it is harder to achieve strong stability. 

\subsection{Proof of \cref{lem-key}}
We now prove \cref{lem-key}. Instead of using the mathematical induction, we provide a more motivated proof below. Let $\cE_s(\tau) = \nmH{R_s(\tau)}^2$. The idea is to use $\frac{d^m}{d\tau^m}\cE(0)$ to determine coefficients in the energy equality. For clearness of the presentation, we first consider the case $Q_{s-(p+1)} = 0$ and then move on to general RK methods. As a convention, matrices and coefficients with negative or fractional indexes are considered as $0$.

\begin{proof}
	\textbf{Step 1: ($R_s = P_p.$)} \smallskip
	
	It is easy to check $\beta_0 =\alpha_0^2 = 1$.
	For $m>0$, using the fact that 
	\begin{equation}
	\ddtau{m} \ipH{v,w} = \sum_{k = 0}^m \binom{m}{k}\ipH{ \ddtau{k}v, \ddtau{m-k}w}
	\end{equation}	
	and
	\begin{equation}
	\dtau P_p = P_{p-1} L,
	\end{equation}
	we have
	\begin{equation}
	\ddtau{m} \cE_p(\tau)  = \sum_{k=0}^m\binom{m}{k}\ipH{P_{p-k} L^k u,P_{p-(m-k)}L^{m-k} u}.
	\end{equation}
	Noting that $P_{p-k}(0) = I$ if $k\leq p$, for $1\leq m \leq p+2$,  one can obtain
	\begin{equation}\label{eq-d1}
	\ddtau{m}\cE_p(0) =-2 \mu + 
	\sum_{k=0}^m\binom{m}{k}\ipH{L^k u,L^{m-k} u} ,
	\end{equation}
	where
	\begin{equation}\label{eq-d12}
	\mu =
	\left\{\begin{array}{cc}
	0,& \qquad 1\leq m\leq p,\\
	\ipH{u,L^{p+1} u},& \qquad m = p+1,\\
	\ipH{u,L^{p+2} u}+ (p+2)\ipH{Lu, L^{p+1} u},& \qquad m = p+2.
	\end{array} \right.
	\end{equation}
	Furthermore, since
	\begin{equation}
	\begin{aligned}
	&\sum_{k=0}^m\binom{m}{k}\ipH{L^k u,L^{m-k} u}\\ 
	=& \sum_{k=1}^m\binom{m-1}{k-1}\ipH{L^k u,L^{m-k} u}+\sum_{k = 0}^{m-1}\binom{m-1}{k}\ipH{L^k u,L^{m-k} u}\\
	=& \sum_{k=0}^{m-1}\binom{m-1}{k}
	\left(\ipH{L^{k+1} u,L^{m-(k+1)} u}+\ipH{L^k u,L^{m-k} u}\right)\\
	=& -\sum_{k=0}^{m-1}\binom{m-1}{k}
	\qipH{L^{k} u,L^{m-1-k} u},\\
	\end{aligned}
	\end{equation}
	\eqref{eq-d1} can be written as
	\begin{equation}\label{eq-d11}
	\ddtau{m}\cE_p(0) =
	-2\mu-\sum_{k=0}^{m-1}\binom{m-1}{k}\qipH{L^{k} u,L^{m-1-k} u}, \qquad 1\leq m\leq p+2.
	\end{equation}
	On the other hand, by differentiating the expansion in \cref{lem-exps}, we obtain
	\begin{equation}\label{eq-d21}
	\ddtau{m}\cE_p(0) = \beta_{\frac{m}{2}} m! \nmH{L^\frac{m}{2} u}^2 + \sum_{ k = 0}^{p-1} \gamma_{k,m-1-k} m!\qipH{L^{k} u, L^{m-1-k}{u}}.
	\end{equation}
	Due to the uniqueness of the expansion,  coefficients in \eqref{eq-d11} and \eqref{eq-d21} must be the same. With $m = 1,\cdots,p$, one can get 
	\begin{equation}
	\beta_{\frac{m}{2}} = 0,\qquad \forall 1\leq m\leq  p
	\end{equation} and 
	\begin{equation}\label{eq-gamma-1}
	\gamma_{i,j} = -\frac{1}{(i+j+1)!}\binom{i+j}{i} = -\frac{1}{i!j!(i+j+1)},\qquad i+j\leq p - 1 ,\quad  i, j \geq 0.
	\end{equation} 
	
	\textbf{Case I: ($p$ is odd.)} If $p$ is odd, then
	\begin{equation}
	\beta_k = 0,\qquad  k = 1,2,\cdots, \frac{p-1}{2}.
	\end{equation}
	We need to compute $\ddtau{p+1}\cE_p(0)$ to determine $\beta_{\frac{p+1}{2}}$. With $m = p+1$, \eqref{eq-d11} and \eqref{eq-d21} imply the following identity. 
	\begin{equation}\label{eq-odd-1}
	\begin{aligned}
	\ddtau{p+1}\cE_p(0) &= \beta_{\frac{p+1}{2}} (p+1)! \nmH{L^{\frac{p+1}{2}} u}^2 + \sum_{ k = 0}^{p-1} \gamma_{k,p-k} (p+1)!\qipH{L^{k} u, L^{p-k}{u}}
	\\ &= -2\ipH{u,L^{p+1}u}-\sum_{k=0}^{p}\binom{p}{k}\qipH{L^{k} u,L^{p-k} u}.
	\end{aligned}
	\end{equation}
	From \cref{cor-induction}, we have
	\begin{equation}\label{eq-odd-2}
	\ipH{u,L^{p+1}u} = (-1)^{\frac{p+1}{2}}\nmH{L^{\frac{p+1}{2}}u}^2-\sum_{k=0}^{\frac{p-1}{2}}(-1)^{k}\qipH{L^k u,L^{p-k} u}.
	\end{equation} 
	Hence \eqref{eq-odd-1} together with \eqref{eq-odd-2} gives
	\begin{equation}\label{eq-odd-bstar}
	\beta_{\frac{p+1}{2}} = -(-1)^{\frac{p+1}{2}} \frac{2}{(p+1)!} \neq 0,
	\end{equation}
	which implies $k^* = \frac{p+1}{2}$. For $0\leq i,j\leq k^*-1$, we have $i + j \leq 2k^* -2 = p-1$. Hence the leading submatrix $\Gamma^*$ is completely determined by \eqref{eq-gamma-1}.\smallskip
	
	\textbf{Case II: ($p$ is even).} For even $p$, 
	\begin{equation}
	\beta_k = 0,\qquad  k = 1,2,\cdots, \frac{p}{2}.
	\end{equation}
	Note $\beta_{\frac{p}{2}+1}$ appears in the expansion of  $\ddtau{p+2}\cE_p(0)$. With $m = p+2$, one can derive from \eqref{eq-d11} and \eqref{eq-d21} that 
	\begin{equation}\label{eq-even-1}
	\begin{aligned}
	&\ddtau{p+2}\cE_p(0)\\
	 =& \beta_{\frac{p}{2}+1} (p+2)! \nmH{L^{\frac{p}{2}+1} u}^2 + \sum_{ k = 0}^{p-1} \gamma_{k,p+1-k} (p+2)!\qipH{L^{k} u, L^{p+1-k}{u}}
	\\ =& -2\ipH{u,L^{p+2}u}-2(p+2)\ipH{Lu,L^{p+1}u}-\sum_{k=0}^{p+1}\binom{p+1}{k}\qipH{L^{k} u,L^{p+1-k} u}.
	\end{aligned}
	\end{equation}
	From \cref{prop-routine}, we have
	\begin{equation}\label{eq-even-2}
	\ipH{u,L^{p+2}u} = (-1)^{\frac{p}{2}+1}\nmH{L^{\frac{p}{2}+1}u}^2-\sum_{k=0}^{\frac{p}{2}}(-1)^{k}\qipH{L^{k} u,L^{{p+1}-k} u}
	\end{equation}
	and
	\begin{equation}\label{eq-even-3}
	\ipH{Lu,L^{p+1}u} = (-1)^{\frac{p}{2}}\nmH{L^{\frac{p}{2}+1}u}^2-\sum_{k=0}^{\frac{p}{2}-1}(-1)^{k}\qipH{L^{k+1} u,L^{{p}-k} u}.
	\end{equation} 
	After identifying the coefficients of $\nmH{L^{\frac{p}{2}+1}u}^2$ in  \eqref{eq-even-1} with \eqref{eq-even-2} and \eqref{eq-even-3}, one can get
	\begin{equation}
	\beta_{\frac{p}{2}+1} =  (-1)^{\frac{p}{2}+1}
	\frac{2}{p!(p+2)}\neq 0.
	\end{equation}
	Hence $k^* = \frac{p}{2}+1$. 
	
	Except for $\gamma_{\frac{p}{2},\frac{p}{2}}$, all other $\gamma_{i,j}$ with $i+j\leq p-1$ are clarified in \eqref{eq-gamma-1}. 
	To determine $\gamma_{\frac{p}{2},\frac{p}{2}}$, we need to consider  $\frac{d^{p+1}}{d\tau^{p+1}}\cE_p(0)$. Note that \eqref{eq-odd-1} holds regardless of the parity of $p$. Furthermore, with $p$ being even, 
	\begin{equation}\label{eq-even-4}
	\ipH{u,L^{p+1}u} = (-1)^{\frac{p}{2}+1} \frac{1}{2}\qnmH{L^\frac{p}{2}u}^2 - \sum_{k = 0}^{\frac{p}{2}-1}(-1)^k\qipH{L^k u , L^{p-k}u}.
	\end{equation}Hence by comparing the coefficient of $\qnmH{L^\frac{p}{2}u}^2$ in \eqref{eq-odd-1} with that in \eqref{eq-even-4}, we get
	\begin{equation}
	\gamma_{\frac{p}{2},\frac{p}{2}} = -\frac{1}{(\frac{p}{2})!(\frac{p}{2})!(p+1)} -(-1)^{\frac{p}{2}+1}\frac{1}{(p+1)!}.
	\end{equation}
	Now we have proved \cref{lem-key} for $R_s = P_p$. \bigskip
	
	\textbf{Step 2: $R_s = P_p + (\tau L)^{p+1} Q_{s-(p+1)}$.}\smallskip
	
	The proof of the general case is based on the fact that not all high order stages will contribute to $k^*$ and $\Gamma^*$. Note that
	\begin{equation}
	\nmH{R_su}^2 = \nmH{P_pu}^2 + \sum_{\max(i,j)>p} \alpha_i \alpha_j \tau^{i+j}\ipH{L^iu,L^ju}.
	\end{equation}
	If $p$ is odd, from \cref{cor-induction}, only $2\alpha_{p+1}\tau^{p+1}\ipH{L^0u,L^{p+1}u}$ (since $\alpha_0 = 1$) in the second term would affect $\beta_{\frac{p+1}{2}}$. With this additional term being considered, instead of  \eqref{eq-odd-1}, we obtain
	\begin{equation}\label{eq-general-1}
	\begin{aligned}
	\ddtau{p+1}\cE_s(0) &= \beta_{\frac{p+1}{2}} (p+1)! \nmH{L^{\frac{p+1}{2}} u}^2 + \sum_{ k = 0}^{p-1} \gamma_{k,p-k} (p+1)!\qipH{L^{k} u, L^{p-k}{u}}
	\\ &= (2\alpha_{p+1}(p+1)!-2)\ipH{u,L^{p+1}u}-\sum_{k=0}^{p}\binom{p}{k}\qipH{L^{k} u,L^{p-k} u}.
	\end{aligned}
	\end{equation}
	After expanding $\langle u, L^{p+1} u \rangle$ with \eqref{eq-odd-2}, \eqref{eq-general-1} implies
	\begin{equation}
	\beta_{\frac{p+1}{2}} =(-1)^\frac{p+1}{2}2\left( \alpha_{p+1}-\frac{1}{(p+1)!}\right).
	\end{equation} 
	Since $\alpha_{p+1} \neq \frac{1}{(p+1)!}$,  we have $\beta_{\frac{p+1}{2}} \neq 0$ and $k^{*} = \frac{p+1}{2}$. 
	Also note that $\ipH{L^iu,L^ju}$ can only produce $\qipH{L^{i'}u,L^{j'}u}$ terms with $i'+j' = i+j-1$. Hence if $\max\{i,j\}>p$, $\ipH{L^iu,L^ju}$ can not affect values of $\gamma_{i',j'}$ with $i'+j'\leq 2(k^* -1) = p-1$. In other words, the leading submatrix $\Gamma^*$ is unchanged.
	
	Similarly, for even $p$, we still have $\beta_1 = \cdots = \beta_{\frac{p}{2} } = 0$ after adding extra high order terms, which implies $k^*\geq \frac{p}{2}+1$. When computing $\frac{d^{p+2}}{d\tau^{p+2}}\cE(0)$ to obtain $\beta_{\frac{p}{2}+1}$, an extra term $2\tau^{p+2}\left(\alpha_{p+2}\ipH{L^0u,L^{p+2}}+\alpha_{p+1}\ipH{L^1u,L^{p+1}u}\right)$ should be considered. 
	Then we have
	\begin{equation}\label{eq-general-2}
	\begin{aligned}
	&\ddtau{p+2}\cE_s(0) \\
	=& \beta_{\frac{p}{2}+1} (p+2)! \nmH{L^{\frac{p}{2}+1} u}^2 + \sum_{ k = 0}^{p-1} \gamma_{k,p+1-k} (p+2)!\qipH{L^{k} u, L^{p+1-k}{u}}
	\\ =& 2\left(\alpha_{p+2}(p+2)!-1\right)\ipH{u,L^{p+2}u}+2\left(\alpha_{p+1}(p+2)!-(p+2)\right)\ipH{Lu,L^{p+1}u}\\
	&-\sum_{k=0}^{p+1}\binom{p+1}{k}\qipH{L^{k} u,L^{p+1-k} u}.
	\end{aligned}
	\end{equation}
	After substituting \eqref{eq-even-2} and \eqref{eq-even-3} into \eqref{eq-general-2}, we get
	\begin{equation}
	\beta_{\frac{p}{2}+1} = (-1)^{\frac{p}{2}+1}2\left(\alpha_{p+2}-\alpha_{p+1}+\frac{1}{p!(p+2)}\right).
	\end{equation}
	As for the corresponding principle submatrix, only $\gamma_{\frac{p}{2},\frac{p}{2}}$ will be changed due to $2 \alpha_{p+1}\tau^{p+1}\ipH{L^0u,L^{p+1}u}$. Then with \eqref{eq-general-1} and \eqref{eq-even-4}, one can obtain
	\begin{equation}
	\gamma_{\frac{p}{2},\frac{p}{2}} = -\frac{1}{(\frac{p}{2})!(\frac{p}{2})!(p+1)} + (-1)^{\frac{p}{2}+1}\left(\alpha_{p+1}-\frac{1}{(p+1)!}\right),
	\end{equation}
	which completes the proof.
\end{proof}

\begin{remark}[Uniqueness in \cref{lem-exps}]
	Without showing the uniqueness in \cref{lem-exps}, one can still prove \cref{lem-key} with induction and then derive other results 
	in \cref{sec-4}. The reason for justifying the uniqueness, is to exclude the possibility that the leading submatrices for even order RK methods are negative definite under another expansion, which may result in an if and only if condition, as that for the methods of odd linear order in \cref{thm-odd}.  
\end{remark}

\section{Conclusions}\label{sec-5}
In this paper, we present a framework on analyzing the strong stability of explicit RK methods for solving semi-negative linear autonomous systems, which are typically obtained from stable method of lines schemes for  hyperbolic problems. With this framework, strong stability in one step or multiple steps of various RK methods are examined. Finally, we analyze the coefficients in the energy equality, based on which, corollaries regarding the criteria of strong stability are derived.

\bibliographystyle{siamplain}

\begin{thebibliography}{10}
	
	\bibitem{bogacki1996efficient}
	{\sc P.~Bogacki and L.~F. Shampine}, {\em An efficient {R}unge--{K}utta (4, 5)
		pair}, Computers \& Mathematics with Applications, 32 (1996), pp.~15--28.
	
	\bibitem{butcher2016numerical}
	{\sc J.~C. Butcher}, {\em Numerical Methods for Ordinary Differential
		Equations}, John Wiley \& Sons, 2016.
	
	\bibitem{chavent1989local}
	{\sc G.~Chavent and B.~Cockburn}, {\em The local projection $
		p^0-p^1$-discontinuous-{G}alerkin finite element method for scalar
		conservation laws}, ESAIM: Mathematical Modelling and Numerical Analysis, 23
	(1989), pp.~565--592.
	
	\bibitem{cheng2017l2}
	{\sc Y.~Cheng, C.-S. Chou, F.~Li, and Y.~Xing}, {\em ${L}^2$ stable
		discontinuous {G}alerkin methods for one-dimensional two-way wave equations},
	Mathematics of Computation, 86 (2017), pp.~121--155.
	
	\bibitem{chung2009optimal}
	{\sc E.~T. Chung and B.~Engquist}, {\em Optimal discontinuous {G}alerkin
		methods for the acoustic wave equation in higher dimensions}, SIAM Journal on
	Numerical Analysis, 47 (2009), pp.~3820--3848.
	
	\bibitem{rkdg4}
	{\sc B.~Cockburn, S.~Hou, and C.-W. Shu}, {\em The {R}unge--{K}utta local
		projection discontinuous {G}alerkin finite element method for conservation
		laws. {I}{V}. the multidimensional case}, Mathematics of Computation, 54
	(1990), pp.~545--581.
	
	\bibitem{rkdg3}
	{\sc B.~Cockburn, S.-Y. Lin, and C.-W. Shu}, {\em {T}{V}{B} {R}unge--{K}utta
		local projection discontinuous {G}alerkin finite element method for
		conservation laws {I}{I}{I}: one-dimensional systems}, Journal of
	Computational Physics, 84 (1989), pp.~90--113.
	
	\bibitem{rkdg2}
	{\sc B.~Cockburn and C.-W. Shu}, {\em {T}{V}{B} {R}unge--{K}utta local
		projection discontinuous {G}alerkin finite element method for conservation
		laws. {I}{I}. general framework}, Mathematics of computation, 52 (1989),
	pp.~411--435.
	
	\bibitem{rkdg1}
	{\sc B.~Cockburn and C.-W. Shu}, {\em The {R}unge--{K}utta local projection $
		{P}^1$-discontinuous-{G}alerkin finite element method for scalar conservation
		laws}, ESAIM: Mathematical Modelling and Numerical Analysis, 25 (1991),
	pp.~337--361.
	
	\bibitem{rkdg5}
	{\sc B.~Cockburn and C.-W. Shu}, {\em The {R}unge--{K}utta discontinuous
		{G}alerkin method for conservation laws {V}: multidimensional systems},
	Journal of Computational Physics, 141 (1998), pp.~199--224.
	
	\bibitem{cockburn2001runge}
	{\sc B.~Cockburn and C.-W. Shu}, {\em {R}unge--{K}utta discontinuous {G}alerkin
		methods for convection-dominated problems}, Journal of scientific computing,
	16 (2001), pp.~173--261.
	
	\bibitem{fu2018optimal}
	{\sc G.~Fu and C.-W. Shu}, {\em Optimal energy-conserving discontinuous
		{G}alerkin methods for linear symmetric hyperbolic systems}, arXiv preprint
	arXiv:1804.10307,  (2018).
	
	\bibitem{fu2019energy}
	{\sc G.~Fu and C.-W. Shu}, {\em An energy-conserving ultra-weak discontinuous
		{G}alerkin method for the generalized {K}orteweg--de {V}ries equation},
	Journal of Computational and Applied Mathematics, 349 (2019), pp.~41--51.
	
	\bibitem{gottlieb1998total}
	{\sc S.~Gottlieb and C.-W. Shu}, {\em Total variation diminishing
		{R}unge--{K}utta schemes}, Mathematics of Computation, 67 (1998), pp.~73--85.
	
	\bibitem{gottlieb2001strong}
	{\sc S.~Gottlieb, C.-W. Shu, and E.~Tadmor}, {\em Strong stability-preserving
		high-order time discretization methods}, SIAM review, 43 (2001), pp.~89--112.
	
	\bibitem{gustafsson1995time}
	{\sc B.~Gustafsson, H.-O. Kreiss, and J.~Oliger}, {\em Time Dependent Problems
		and Difference Methods}, John Wiley \& Sons, 1995.
	
	\bibitem{iserles2009first}
	{\sc A.~Iserles}, {\em A First Course in the Numerical Analysis of Differential
		Equations}, Cambridge University Press, 2009.
	
	\bibitem{kraaijevanger1991contractivity}
	{\sc J.~F. B.~M. Kraaijevanger}, {\em Contractivity of {R}unge--{K}utta
		methods}, BIT Numerical Mathematics, 31 (1991), pp.~482--528.
	
	\bibitem{randall1992numerical}
	{\sc R.~J. LeVeque}, {\em Numerical Methods for Conservation Laws},
	Birkh\"{a}user, 1992.
	
	\bibitem{levy1998semidiscrete}
	{\sc D.~Levy and E.~Tadmor}, {\em From semidiscrete to fully discrete:
		{S}tability of {R}unge--{K}utta schemes by the energy method}, SIAM review,
	40 (1998), pp.~40--73.
	
	\bibitem{qin2018implicit}
	{\sc T.~Qin and C.-W. Shu}, {\em Implicit positivity-preserving high-order
		discontinuous {G}alerkin methods for conservation laws}, SIAM Journal on
	Scientific Computing, 40 (2018), pp.~A81--A107.
	
	\bibitem{qin2016bound}
	{\sc T.~Qin, C.-W. Shu, and Y.~Yang}, {\em Bound-preserving discontinuous
		{G}alerkin methods for relativistic hydrodynamics}, Journal of Computational
	Physics, 315 (2016), pp.~323--347.
	
	\bibitem{ranocha2018l_2}
	{\sc H.~Ranocha and P.~{\"O}ffner}, {\em ${L}_2$ stability of explicit
		{R}unge--{K}utta schemes}, Journal of Scientific Computing,  (2018),
	pp.~1--17.
	
	\bibitem{shampine2018numerical}
	{\sc L.~F. Shampine}, {\em Numerical Solution of Ordinary Differential
		Equations}, Routledge, 2018.
	
	\bibitem{shu1988total}
	{\sc C.-W. Shu}, {\em Total-variation-diminishing time discretizations}, SIAM
	Journal on Scientific and Statistical Computing, 9 (1988), pp.~1073--1084.
	
	\bibitem{shu1988efficient}
	{\sc C.-W. Shu and S.~Osher}, {\em Efficient implementation of essentially
		non-oscillatory shock-capturing schemes}, Journal of computational physics,
	77 (1988), pp.~439--471.
	
	\bibitem{sofroniou2004construction}
	{\sc M.~Sofroniou and G.~Spaletta}, {\em Construction of explicit
		{R}unge-{K}utta pairs with stiffness detection}, Mathematical and Computer
	Modelling, 40 (2004), pp.~1157--1169.
	
	\bibitem{spijker1983contractivity}
	{\sc M.~Spijker}, {\em Contractivity in the numerical solution of initial value
		problems}, Numerische Mathematik, 42 (1983), pp.~271--290.
	
	\bibitem{spiteri2002new}
	{\sc R.~J. Spiteri and S.~J. Ruuth}, {\em A new class of optimal high-order
		strong-stability-preserving time discretization methods}, SIAM Journal on
	Numerical Analysis, 40 (2002), pp.~469--491.
	
	\bibitem{sun2018discontinuous}
	{\sc Z.~Sun, J.~A. Carrillo, and C.-W. Shu}, {\em A discontinuous {G}alerkin
		method for nonlinear parabolic equations and gradient flow problems with
		interaction potentials}, Journal of Computational Physics, 352 (2018),
	pp.~76--104.
	
	\bibitem{sun2018entropy}
	{\sc Z.~Sun, J.~A. Carrillo, and C.-W. Shu}, {\em An entropy stable high-order
		discontinuous {G}alerkin method for cross-diffusion gradient flow systems},
	arXiv preprint arXiv:1810.03221,  (2018).
	
	\bibitem{sun2017stability}
	{\sc Z.~Sun and C.-W. Shu}, {\em Stability analysis and error estimates of
		{L}ax--{W}endroff discontinuous {G}alerkin methods for linear conservation
		laws}, ESAIM: Mathematical Modelling and Numerical Analysis, 51 (2017),
	pp.~1063--1087.
	
	\bibitem{sun2017rk4}
	{\sc Z.~Sun and C.-W. Shu}, {\em Stability of the fourth order {R}unge--{K}utta
		method for time-dependent partial differential equations}, Annals of
	Mathematical Sciences and Applications, 2 (2017), pp.~255--284.
	
	\bibitem{tadmor2002semidiscrete}
	{\sc E.~Tadmor}, {\em From semidiscrete to fully discrete: Stability of
		{R}unge--{K}utta schemes by the energy method. {I}{I}}, Collected Lectures on
	the Preservation of Stability under Discretization, Lecture Notes from
	Colorado State University Conference, Fort Collins, CO, 2001 (D. Estep and S.
	Tavener, eds.), Proceedings in Applied Mathematics, SIAM, 109 (2002),
	pp.~25--49.
	
	\bibitem{verner2010numerically}
	{\sc J.~H. Verner}, {\em Numerically optimal {R}unge--{K}utta pairs with
		interpolants}, Numerical Algorithms, 53 (2010), pp.~383--396.
	
	\bibitem{wang2015stability}
	{\sc H.~Wang, C.-W. Shu, and Q.~Zhang}, {\em Stability and error estimates of
		local discontinuous {G}alerkin methods with implicit-explicit time-marching
		for advection-diffusion problems}, SIAM Journal on Numerical Analysis, 53
	(2015), pp.~206--227.
	
	\bibitem{wang2016stability}
	{\sc H.~Wang, C.-W. Shu, and Q.~Zhang}, {\em Stability analysis and error
		estimates of local discontinuous {G}alerkin methods with implicit-explicit
		time-marching for nonlinear convection-diffusion problems}, Applied
	Mathematics and Computation, 272 (2016), pp.~237--258.
	
	\bibitem{wolfram2008advanced}
	{\sc S.~Wolfram}, {\em Advanced numerical differential equation solving in
		{M}athematica}, Wolfram Mathematica Tutorial Collection. Wolfram Research,
	Champaign,  (2008).
	
	\bibitem{wu2015high}
	{\sc K.~Wu and H.~Tang}, {\em High-order accurate
		physical-constraints-preserving finite difference {W}{E}{N}{O} schemes for
		special relativistic hydrodynamics}, Journal of Computational Physics, 298
	(2015), pp.~539--564.
	
	\bibitem{xing2010positivity}
	{\sc Y.~Xing, X.~Zhang, and C.-W. Shu}, {\em Positivity-preserving high order
		well-balanced discontinuous {G}alerkin methods for the shallow water
		equations}, Advances in Water Resources, 33 (2010), pp.~1476--1493.
	
	\bibitem{zhang2012fully}
	{\sc Q.~Zhang and F.~Gao}, {\em A fully-discrete local discontinuous {G}alerkin
		method for convection-dominated {S}obolev equation}, Journal of Scientific
	Computing, 51 (2012), p.~107–134.
	
	\bibitem{zhang2010stability}
	{\sc Q.~Zhang and C.-W. Shu}, {\em Stability analysis and a priori error
		estimates of the third order explicit {R}unge--{K}utta discontinuous
		{G}alerkin method for scalar conservation laws}, SIAM Journal on Numerical
	Analysis, 48 (2010), pp.~1038--1063.
	
	\bibitem{zhang2010maximum}
	{\sc X.~Zhang and C.-W. Shu}, {\em On maximum-principle-satisfying high order
		schemes for scalar conservation laws}, Journal of Computational Physics, 229
	(2010), pp.~3091--3120.
	
	\bibitem{zhang2010positivity}
	{\sc X.~Zhang and C.-W. Shu}, {\em On positivity-preserving high order
		discontinuous {G}alerkin schemes for compressible {E}uler equations on
		rectangular meshes}, Journal of Computational Physics, 229 (2010),
	pp.~8918--8934.
	
	\bibitem{zhang2011maximum}
	{\sc X.~Zhang and C.-W. Shu}, {\em Maximum-principle-satisfying and
		positivity-preserving high-order schemes for conservation laws: survey and
		new developments}, Proceedings of the Royal Society of London A:
	Mathematical, Physical and Engineering Sciences, 467 (2011), pp.~2752--2776.
	
	\bibitem{zhou2018stability}
	{\sc L.~Zhou, Y.~Xia, and C.-W. Shu}, {\em Stability analysis and error
		estimates of arbitrary {L}agrangian--{E}ulerian discontinuous {G}alerkin
		method coupled with {R}unge--{K}utta time-marching for linear conservation
		laws}, ESAIM: Mathematical Modelling and Numerical Analysis,  (2018, to
	appear).
	
\end{thebibliography}

\end{document}